\documentclass[11pt]{article}
\usepackage{amsmath,amsthm,amssymb,amsfonts}
\usepackage{latexsym}
\usepackage{graphicx,psfrag,import}
\usepackage{fullpage}
\usepackage{framed}
\usepackage{verbatim}
\usepackage{color}
\usepackage{epsfig}
\usepackage{epstopdf}
\usepackage{hyperref}
\usepackage{geometry}
\usepackage{mathtools}
\usepackage{arydshln}
\usepackage{enumerate}
\usepackage{multicol}
\usepackage{a4wide}
\usepackage{booktabs}
\usepackage{lineno}
\usepackage{parcolumns}
\usepackage{thmtools}
\usepackage{xr}
\usepackage{epstopdf}
\usepackage{mathrsfs}
\usepackage{subfig}
\usepackage{caption}
\usepackage{comment}
\usepackage{authblk}
\usepackage{setspace}
\usepackage{algorithm}
\usepackage[normalem]{ulem}
\usepackage{cleveref}

\usepackage{footmisc}

\theoremstyle{plain}
\newtheorem{theorem}{Theorem}[section]

\newtheorem{proposition}[theorem]{Proposition}
\newtheorem{lemma}[theorem]{Lemma}

\theoremstyle{definition}
\newtheorem{definition}[theorem]{Definition}
\newtheorem{example}[theorem]{Example}

\newtheorem{remark}[theorem]{Remark}

\theoremstyle{remark}

\newcommand\RR{\mathbb{R}}

\newcommand\by{\boldsymbol{y}}

\newcommand\byi{\boldsymbol{y_i}}

\newcommand\byj{\boldsymbol{y_j}}

\renewcommand\bf{\boldsymbol{f}}

\newcommand\bk{\boldsymbol{k}}

\newcommand\bx{\boldsymbol{x}}

\newcommand\bJ{\boldsymbol{J}}

\newcommand{\mK}{\mathcal{K}}
\newcommand{\mKe}{\mK_{\RR}}
\newcommand{\mKd}{\mK_{\rm{disg}}}
\newcommand{\dK}{\mathcal{K}_{\RR\text{-disg}}}
\newcommand{\mJ}{\mathcal{J}}
\newcommand{\mJe}{\mJ_{\RR}}

\newcommand{\mS}{\mathcal{S}}
\newcommand{\mSG}{\mathcal{S}_G}

\DeclareMathOperator{\spn}{span}

        {\begin{list}
                {\noindent\makebox[0mm][r]{$\bullet$}}
                {\leftmargin=5.5ex \usecounter{enumi}

      \topsep=1.5mm \itemsep=-.75ex}
        }
        {\end{list}}

\begin{document}

\title{On the Connectivity of the Disguised Toric Locus of a Reaction Network}

\author[1]{
         Gheorghe Craciun%
}
\author[2]{
        Abhishek Deshpande%
}
\author[3]{
        Jiaxin Jin%
}
\affil[1]{\small Department of Mathematics and Department of Biomolecular Chemistry, University of Wisconsin-Madison}
\affil[2]{Center for Computational Natural Sciences and Bioinformatics, \protect \\
 International Institute of Information Technology Hyderabad}
\affil[3]{\small Department of Mathematics, The Ohio State University}

\date{}

\maketitle

\begin{abstract}
\noindent
Complex-balanced mass-action systems are some of the most important types of mathematical models of reaction networks, due  to their widespread use in applications, as well as their remarkable stability properties. We study the set of positive parameter values (i.e., reaction rate constants) of a reaction network $G$ that, according to mass-action kinetics,  generate dynamical systems that can be realized as complex-balanced systems, possibly by using a different graph $G'$. This set of parameter values is called  the \emph{disguised toric locus} of $G$. 
The $\RR$-\emph{disguised toric locus} of $G$ is defined analogously, except that the parameter values are allowed to take on any real values. We  prove that the disguised toric locus of $G$ is path-connected, and the $\RR$-disguised toric locus of $G$ is also path-connected. We also show that the closure of the disguised toric locus of a reaction network contains the union of the disguised toric loci of all its subnetworks.
\end{abstract}

\tableofcontents

\section{Introduction}

Dynamical systems generated by reaction networks can exhibit a very wide range of complex dynamic behaviors, such as multistability,  limit cycles, chaotic dynamics, but also persistence and global stability~\cite{yu2018mathematical}. A particular class of reaction systems called {\em complex-balanced systems} are known to display very stable behavior. In particular, complex-balanced  systems admit strictly convex Lyapunov functions that guarantee local asymptotic stability of their positive steady states. Further, these dynamical systems are conjectured to be globally stable; this property has been proved under some additional assumptions~\cite{yu2018mathematical, anderson2011proof, gopalkrishnan2014geometric, craciun2015toric, feinberg2019foundations}. 

An object of particular interest is the set in the space of network parameters that generates complex-balanced dynamical systems. This set is called the \emph{toric locus} and has been studied in depth using linear algebra and  algebraic geometry, see for example~\cite{CraciunDickensteinShiuSturmfels2009}. There exists a larger set in the space of network parameters that generates dynamical systems that can be \emph{realized} by complex-balanced systems. This set is called the \emph{disguised toric locus}. In this paper, we focus on the disguised toric locus. 
In particular, we  show that the disguised toric locus is path-connected. Note that other important subsets of the parameter space of a network have also been studied; in particular, the connectivity of the {\em multistationarity locus} was investigated recently in~\cite{telek_feliu_2023}.

Our paper is structured as follows: In Section~\ref{sec:reaction_network} we introduce reaction networks and some basic terminology associated with them. In Section~\ref{sec:complex_balanced} we introduce the notions of toric dynamical systems and flux systems. In Section~\ref{sec:toric_locus}, we introduce the \emph{toric locus} and analyze its properties. In particular, we establish some homeomorphisms in the context of the toric locus which imply that the toric locus and its subsets have a product structure. In addition, we show that given a reaction network, the union of the disguised toric loci of all its subnetworks is contained in the closure of the disguised toric locus of the original network. In Section~\ref{sec:main_result}, we introduce the \emph{disguised toric locus} of a reaction network and prove that it is path-connected. In Section~\ref{sec:discussion}, we recapitulate our main results and chalk out directions for future research.

\medskip

\textbf{Notation.} We will use the following notation throughout the paper:
\begin{itemize}

\item $\mathbb{R}_{\geq 0}^n$ and $\mathbb{R}_{>0}^n$: the set of vectors in $\mathbb{R}^n$ with non-negative (resp. positive) entries.


\item Given two vectors $\bx = (x_1, \ldots, x_n)^{\intercal}\in \RR^n_{>0}$ and $\by = (y_1, \ldots, y_n)^{\intercal} \in \RR^n$, we define:
\begin{equation} \notag
\bx^{\by} = x_1^{y_{1}} x_2^{y_{2}} \ldots x_n^{y_{n}}.
\end{equation}

\end{itemize}


\section{Background}
\label{sec:background}

The goal of this section is to recall some terminology related to reaction networks.

\subsection{Reaction Networks}
\label{sec:reaction_network}

\begin{definition}[\cite{craciun2015toric,craciun2020endotactic,craciun2019polynomial}]
\begin{enumerate}[(a)] 
\item A reaction network $G = (V, E)$, also called a \emph{Euclidean embedded graph} (or an \emph{E-graph}), is a directed graph in $\RR^n$, where $V \subset \mathbb{R}^n$ represents a finite set of \emph{vertices} and $E \subseteq V \times V$ represents the set of \emph{edges}, and such that there are no isolated vertices and no self-loops.

\item Given a reaction network $G = (V, E)$, an  edge $(\by, \by') \in E$, also denoted by $\by \to \by'$, is called a {\em reaction} in the   network. For every reaction $\by \rightarrow \by'\in E$, the vertex $\by$ is called the \emph{source vertex}, and the vertex $\by'$ is called the \emph{target vertex}. Further, we refer to the vector~$\by' - \by$ as the {\em reaction vector} of this reaction.

\item The \emph{stoichiometric subspace} $S_G$ of the reaction network $G = (V, E)$ is the linear subspace generated by its reaction vectors, i.e., 
\begin{eqnarray}
S_G : = \spn \{ \by' - \by: \by \rightarrow \by' \in E \}.
\end{eqnarray}
\end{enumerate}
\end{definition}

\begin{definition}
Let $G=(V, E)$ and $\tilde{G} = (\tilde{V}, \tilde{E})$ be two E-graphs.
\begin{enumerate}[(a)] 
\item
A connected component of $G$ is said to be \emph{strongly connected} if every edge in that component is part of an oriented cycle. Further, $G = (V, E)$ is \emph{weakly reversible} if all connected components of $G$ are strongly connected. 

\item $G=(V, E)$ is a \emph{complete graph} if for every pair of distinct vertices $\by, \by' \in V$, $\by \rightarrow \by' \in E$.

\item $G$ is a \emph{subgraph} of $\tilde{G}$ (denoted by $G \subseteq \tilde{G}$) if $V \subseteq \tilde{V}$ and $E \subseteq \tilde{E}$. Further, we let $G \sqsubseteq \tilde{G}$ denote that $G$ is weakly reversible and $G \subseteq \tilde{G}$.
\end{enumerate}
\end{definition}

Given any E-graph $G = (V, E)$, a complete graph (denoted by $G_{\rm{comp}}$) can be obtained by connecting every pair of distinct vertices of $G$. By Definition, $G \subseteq G_{\rm{comp}}$. Further, $G \sqsubseteq G_{\rm{comp}}$ if $G$ is weakly reversible.


\begin{example}

Figure~\ref{fig:e-graph} shows a few examples of reaction networks represented as E-graphs.

\begin{figure}[!ht]
\centering
\includegraphics[scale=0.4]{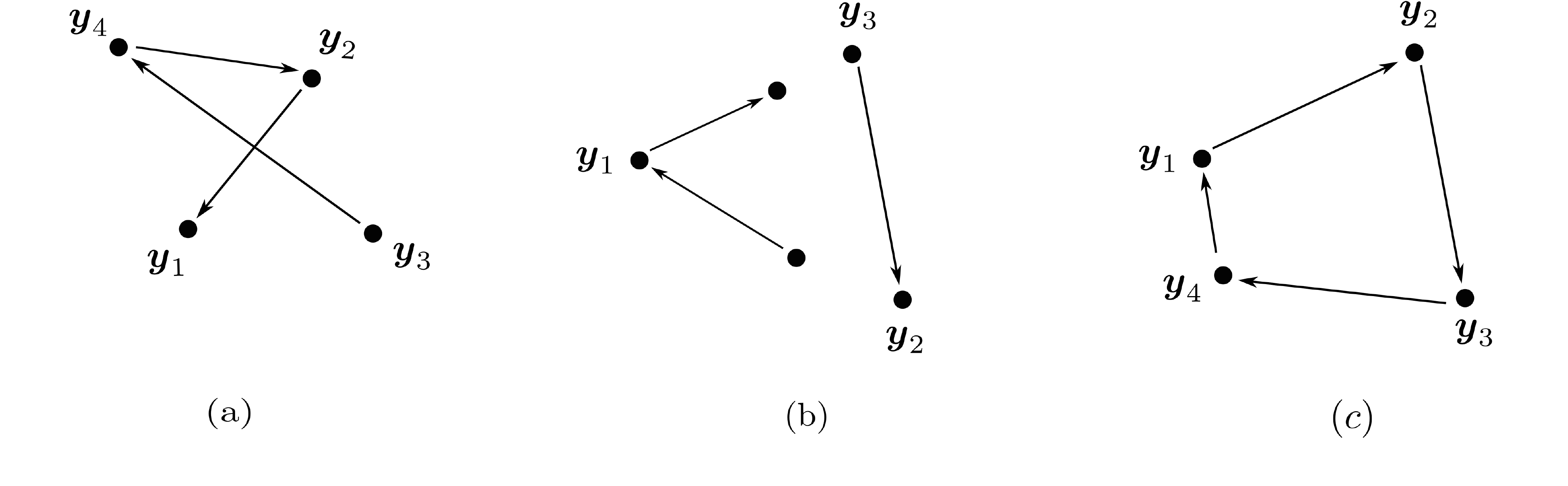}
\caption{\small (a) This reaction network consists of a single connected component. (b) This reaction network consists of two connected components. (c) This reaction network is weakly reversible and contains one connected component. The stoichiometric subspace corresponding to all three networks is $\mathbb{R}^2$.}
\label{fig:e-graph}
\end{figure} 
\qed
\end{example}


We now turn our attention to the dynamics exhibited by a network. 

\begin{definition}[\cite{yu2018mathematical,feinberg1979lectures,voit2015150,guldberg1864studies,gunawardena2003chemical,adleman2014mathematics}]
Let $G=(V, E)$ be an E-graph.
For each $\by\rightarrow \by' \in E$, let $k_{\by \rightarrow \by'} >0$ be the associated reaction rate constant and let $\bk = (k_{\by \rightarrow \by'})_{\by\rightarrow \by' \in E} \in \mathbb{R}_{>0}^{E}$ be the reaction rate vector. Under mass-action kinetics, the dynamical system generated by $(G, \bk)$ is 
\begin{equation} \label{eq:mass_action}
 \frac{\mathrm{d} \bx}{\mathrm{d} t} 
= \sum_{\by\rightarrow \by' \in E}k_{\by \rightarrow \by'} \bx^{\by}(\by' - \by).
\end{equation}
\end{definition}

Given a point $\bx_0 \in \mathbb{R}^n_{>0}$, the \emph{stoichiometric compatibility class} of $\bx_0$ is given by
\begin{eqnarray}
S_{\bx_0} : = (\bx_0 + S_G)\cap\mathbb{R}^n_{>0}.
\end{eqnarray}
If the positive orthant is forward invariant, the stoichiometric compatibility class is an invariant polyhedron~\cite{sontag2001structure}.

\begin{definition} 
\label{def:mas_realizable}
Consider the following dynamical system
\begin{equation} \label{eq:realization_ode}
 \frac{\mathrm{d} \bx}{\mathrm{d} t} 
= \bf (\bx).
\end{equation}
We say the dynamical system is \emph{$\RR$-realizable} on an E-graph $G=(V,E)$, if there exists a rate vector $\bk \in \mathbb{R}^{E}$, such that 
\begin{equation} \label{eq:realization}
\bf (\bx) =
\sum_{\by \rightarrow \by'\in E} k_{\by \rightarrow \by'} \bx^{\by}(\by' - \by).
\end{equation}
Further, if $\bk \in \mathbb{R}^{E}_{>0}$, this dynamical system is said to be \emph{{realizable} on $G$.}
\end{definition}

\begin{definition}[\cite{horn1972general,craciun2008identifiability,anderson2020classes}]
\label{def:de}
Two mass-action systems $(G, \bk)$ and $(\tilde{G}, \tilde{\bk})$ are said to be \emph{dynamically equivalent} (denoted by 
$(G, \bk) \sim (\tilde{G}, \tilde{\bk})$), if for every vertex\footnote{\label{footnote1} Note that when $\by_0 \not\in V$ or $\by_0 \not\in \tilde{V}$, that side is considered as an empty sum, which is zero.} $\by_0 \in V \cup \tilde{V}$,
\begin{equation}
\label{eq:DE}
\sum_{\by_0 \to \by \in E} k_{\by_0  \to \by} (\by - \by_0) 
= \sum_{\by_0 \to \by' \in \tilde{E}} \tilde{k}_{\by_0  \to \by'}  (\by' - \by_0).
\end{equation}
\end{definition} 

\begin{remark}
Suppose $(G, \bk)$ and $(\tilde{G}, \tilde{\bk})$ are two dynamically equivalent mass-action systems. Then $(G, \bk)$ is realizable on $\tilde{G}$ and $(\tilde{G}, \tilde{\bk})$ is realizable on $G$.
\end{remark}


\subsection{Toric Dynamical Systems and the Disguised Toric Locus}
\label{sec:complex_balanced}

The goal of this section is to recall some properties of \emph{complex-balanced systems} (also known as \emph{toric dynamical systems}),
and then define the \emph{disguised toric locus}, which is the set of the reaction rate vectors that allow complex-balanced realizations under dynamical equivalence.

\begin{definition} 
\label{def:cb_system}

Consider the mass-action system $(G, \bk)$ as follows 
\begin{equation} \notag
 \frac{\mathrm{d} \bx}{\mathrm{d} t} 
= \sum_{\by\rightarrow \by' \in E}k_{\by \rightarrow \by'} \bx^{\by}(\by' - \by).
\end{equation}
A point $\bx^* \in \mathbb{R}_{>0}^n$ is called a \emph{positive steady state} of $(G, \bk)$ if it satisfies
\[ 
\sum_{\by \rightarrow \by' \in E}k_{\by \rightarrow \by'} (\bx^*)^{\by}(\by' - \by)
= \mathbf{0}.
\]
A positive steady state $\bx^* \in \RR_{>0}^n$ is called a \emph{complex-balanced steady state} of $(G, \bk)$, if for every vertex $\by \in V$, 
\begin{eqnarray} \notag
\sum_{\by \to \by' \in E} k_{\by \to \by'} (\bx^*)^{\by}
= \sum_{\by' \to \by \in E}k_{\by' \to \by}
(\bx^*)^{\by'}.
\end{eqnarray}
If the mass-action system $(G, \bk)$ admits a complex-balanced steady state, then it is called a \emph{complex-balanced system} or a \emph{toric dynamical system}.
\end{definition}



\begin{definition} 
Let $G=(V, E)$ be an E-graph. The \emph{toric locus} on $G$ is given by 
\begin{equation} \notag
\mK (G) := \{ \bk \in \mathbb{R}_{>0}^{E} \ \big| \ \text{the mass-action system generated by } (G, \bk) \ \text{is complex-balanced} \}.
\end{equation}
Further, for any E-graph $\tilde{G} =(\tilde{V}, \tilde{E})$, we define the set $\mKe(\tilde{G},G)$ as
\begin{equation} \notag
\mKe (\tilde{G},G) := \{ \tilde{\bk} \in \mK (\tilde{G}) \ \big| \ \text{the mass-action system } (\tilde{G}, \tilde{\bk}) \ \text{is $\RR$-realizable on } G \}.
\end{equation}
We also define the set $\mK(\tilde{G},G)$ as
\begin{equation} \notag
\mK (\tilde{G},G) := \{ \tilde{\bk} \in \mK (\tilde{G}) \ \big| \ \text{the mass-action system } (\tilde{G}, \tilde{\bk}) \ \text{is realizable on } G \}.
\end{equation}
From the definition, it is clear that $\mK (G, \tilde{G}) \subset \mKe (G, \tilde{G})$.
\end{definition}


Here we present a useful lemma that connects weakly reversible E-graphs with the toric locus.

\begin{lemma}[\cite{feinberg2019foundations}]
\label{lem:graph_KJ}
Let $G=(V, E)$ be an E-graph. 
If $G=(V,E)$ is weakly reversible, then there exists a rate vector $\bk \in \mathbb{R}_{>0}^{E}$, such that $(G,\bk)$ is complex-balanced, i.e., $\mK (G) \neq \emptyset$.
Otherwise, if $G=(V,E)$ is not weakly reversible, then $\mK (G) = \emptyset$.
\end{lemma}

\begin{definition} 
\label{def:de_realizable}
Let $G =(V,E)$ 
be an E-graph.
\begin{enumerate}[(a)]

\item For any E-graph $\tilde{G} =(\tilde{V}, \tilde{E})$,  define the set $\mKd (G, \tilde{G})$ as
\begin{equation} \notag
\begin{split}
\mKd (G, \tilde{G}) := \{ \bk \in \mathbb{R}^{E}_{>0} \ \big| \ 
& \text{the mass-action system} \ (G, \bk) 
\text{ is dynamically equivalent}
\\& \text{to } (\tilde{G}, \tilde{\bk}) 
\text{ for some } \tilde{\bk} \in \mK (\tilde{G}) \}.
\end{split}
\end{equation}

\item Define the \emph{disguised toric locus} of $G$ as
\begin{equation} \notag
\mK_{\rm{disg}}(G) := \bigcup_{\tilde{G} \sqsubseteq G_{\rm{comp}}} \ \mK_{\rm{disg}}(G, \tilde{G}),
\end{equation}
where the notation $\tilde{G} \sqsubseteq G_{\rm{comp}}$ represents that
$\tilde{G}$ is a weakly reversible subgraph of $G_{\rm{comp}}$.

\item For any E-graph $\tilde{G} =(\tilde{V}, \tilde{E})$,  define the set $\dK (G, \tilde{G})$ as
\begin{equation} \notag
\begin{split}
\dK (G, \tilde{G}) := \{ \bk \in \mathbb{R}^{E} \ \big| \ & \text{the dynamical system generated by } (G, \bk) 
\text{ is dynamically}
\\& \text{equivalent to } (\tilde{G}, \tilde{\bk}) 
\text{ for some } \tilde{\bk} \in \mK (\tilde{G}) \},
\end{split}
\end{equation}
where the dynamical system generated by $(G, \bk)$ is given by 
equation~\eqref{eq:mass_action}\footnote{
For simplicity, in the rest of this paper, we {\em abuse} the following notation:
Given $\bk \in \mathbb{R}_{>0}^{E}$, we will refer to the  mass-action  system generated by $G$ and $\bk$ as in 
equation~\eqref{eq:mass_action}
as ``the mass-action system $(G, \bk)$".
Moreover, we will still refer to this system as ``the mass-action system $(G, \bk)$" even if we have $\bk \in \mathbb{R}^{E}$ instead of $\bk \in \mathbb{R}_{>0}^{E}$.
}. 
Note that $\bk$ here may have non-positive components, and we have $\mKd (G, \tilde{G}) \subset \dK (G, \tilde{G})$.

\item Define the $\RR$-\emph{disguised toric locus} of $G$ as
\begin{equation} \notag
\dK (G) := \bigcup_{\tilde{G} \sqsubseteq G_{\rm{comp}}} \ \dK(G, \tilde{G}).
\end{equation}
where $\tilde{G} \sqsubseteq G_{\rm{comp}}$ represents that
$\tilde{G}$ is a weakly reversible subgraph of $ G_{\rm{comp}}$.
\end{enumerate}
\end{definition}

\begin{remark}
From the definition, it is clear that $\mKd (G) \subseteq \dK(G)$. 
In general, we need both $\mKd (G)$ and $\dK (G)$ to include $\mKd (G,G')$ or $\dK(G, G')$ for {\em any} weakly reversible E-graph $G'$ (i.e., not just for $G' \sqsubseteq G_{\rm{comp}}$). On the other hand, due to results in~\cite{craciun2020efficient}, it turns out that, if a dynamical system generated by $G$ can be realized as toric by some $G'$, then there exists $G'' \sqsubseteq G_{\rm{comp}}$ that {\em also} can give rise to a toric realization of that dynamical system. Therefore, the above assumption that $\tilde{G} \sqsubseteq G_{\rm{comp}}$ still leads to the correct definition. 
\end{remark}

The following is a direct consequence of  the Definition \ref{def:de_realizable}.

\begin{remark}
Consider two E-graphs $G =(V,E)$ and $\tilde{G} =(\tilde{V}, \tilde{E})$. Then 
\begin{enumerate}[(a)]
\item $\mK (\tilde{G}, G)$ is non-empty if and only if $\mKd (G, \tilde{G})$ is non-empty.

\item $\mKe (\tilde{G}, G)$ is non-empty if and only if $\dK (G, \tilde{G})$ is non-empty.
\end{enumerate}
\end{remark}

The sets $\mK (\tilde{G}, G)$ and $\mKe (\tilde{G}, G)$ are related~\cite{disg_1}. In Section~\ref{sec:toric_locus}, we describe some properties of the toric locus. In Section~\ref{sec:main_result}, we show that the disguised toric locus is path-connected.

\section{Toric Locus}
\label{sec:toric_locus}

In this section, we first present some elementary properties of the toric locus. Then we show some homeomorphisms on the toric locus and its subsets.

\subsection{Some Basic Properties of the Toric Locus}
\label{sec:basic_toric_locus}

\begin{lemma}
\label{lem:fiber}
Let $G = (V, E)$ be a weakly reversible E-graph. Consider a reaction rate vector $\bk^* \in \mK(G)$ and a complex-balanced steady state $\bx^* \in \RR^n_{>0}$ for the mass-action system $(G, \bk^*)$. For any $\bx \in (\bx_0 + \mSG )\cap\mathbb{R}^n_{>0}$, we define
\begin{equation} \label{def:fiber_vector}
\bk_{G}^{*} (\bx, \bx^*)
= \left(k_{\by\rightarrow \by'}^{*} (\bx, \bx^*)\right)_{\by\rightarrow \by'\in E}
= \left(  \frac{k^*_{\by\rightarrow \by'}{(\bx^*)}^{\by}}{\bx^{\by}} \right)_{\by\rightarrow \by'\in E}.
\end{equation}
Then 
\begin{equation} \notag
\bk_{G}^{*} (\bx, \bx^*) \in \mK (G),
\end{equation}
and $\bx$ is a complex-balanced steady state for the mass-action system $(G, \bk_{G}^{*} (\bx, \bx^*))$.
\end{lemma}

\begin{proof}
Since $\bx^*$ is a complex-balanced steady state for the mass-action system $(G, \bk^*)$, then for every vertex $\by_0 \in V$,
\begin{equation} \label{eq:fiber_1}
\sum_{\by_0 \to \by' \in E} k^*_{\by_0 \to \by'} (\bx^*)^{\by_0}
= \sum_{\by \to \by_0 \in E} k^*_{\by \to \by_0}
(\bx^*)^{\by}.
\end{equation}
Thus, we derive 
\begin{equation} \label{eq:fiber_2}
\sum_{\by_0 \to \by' \in E} k_{\by_0 \to \by'}^{*} (\bx, \bx^*) (\bx)^{\by_0}
= \sum_{\by_0 \to \by' \in E}  \frac{k^*_{\by_0 \rightarrow \by'}{(\bx^*)}^{\by_0}}{\bx^{\by_0}} (\bx)^{\by_0}
= \sum_{\by_0 \to \by' \in E} k^*_{\by_0 \to \by'} (\bx^*)^{\by_0}.
\end{equation}
On the other hand, we compute that
\begin{equation} \label{eq:fiber_3}
\sum_{\by \to \by_0 \in E} k_{\by \to \by_0}^{*} (\bx, \bx^*) \bx^{\by} 
= \sum_{\by \to \by_0 \in E} \frac{k^*_{\by \to \by_0}{(\bx^*)}^{\by}}{\bx^{\by}} \bx^{\by} 
= \sum_{\by \to \by_0 \in E} k^*_{\by \to \by_0}
(\bx^*)^{\by}.
\end{equation}
Together with~\eqref{eq:fiber_1},~\eqref{eq:fiber_2} and~\eqref{eq:fiber_3}, for every vertex $\by_0 \in V$,
\begin{equation} \notag
\sum_{\by_0 \to \by' \in E} k_{\by_0 \to \by'}^{*} (\bx, \bx^*) (\bx)^{\by_0}
= \sum_{\by \to \by_0 \in E} k_{\by \to \by_0}^{*} (\bx, \bx^*) (\bx)^{\by}.
\end{equation}
Therefore, we conclude that $\bx$ is a complex-balanced steady state for the mass-action system $(G, \bk_{G}^{*} (\bx, \bx^*))$ and $\bk_{G}^{*} (\bx, \bx^*) \in \mK (G)$.
\end{proof}

\medskip

Here we show a relation between the toric locus of a
Euclidean embedded graph and the toric locus of its subgraphs.

\begin{proposition}
\label{prop:subtoric}
Let $G = (V, E)$ be a weakly reversible E-graph. For any subgraph $G_i\subseteq G$, we construct the corresponding set of vectors $\hat{\mK} (G_i)$ as
\begin{equation} \label{subtoric_hat}
\hat{\mK} (G_i) := \{ \hat{\bk} \in \RR_{\geq 0}^E \ | \ \bk \in \mK (G_i) \}
\ \text{ with } \
\hat{k}_{\by \to \by'} :=
\begin{cases}
k_{\by \to \by'}, & \text{if } \by \to \by' \in E_i, \\
0, & \text{if } \by \to \by' \notin E_i.
\end{cases}
\end{equation}
Then we have
\begin{equation} \label{subtoric}
\bigcup_{G_i \subseteq G} \hat{\mK} (G_i) \subsetneq \overline{\mK(G)}.
\end{equation}
\end{proposition}

\begin{proof}
Since the E-graph $G=(V, E)$ is weakly reversible, Lemma \ref{lem:graph_KJ} shows $\mK (G) \neq \emptyset$ and there exists a reaction rate vector $\bk^*\in \mK(G)$. Further, we assume the mass-action system $(G, \bk^*)$ has a complex-balanced steady state $\bx^*$.

Suppose $G_1 = (V_1, E_1)$ is a subgraph of $G$.
If $G_1$ is not weakly reversible, then $\mK (G_1) = \emptyset$ and \eqref{subtoric} holds.
Otherwise, if $G_1$ is weakly reversible, then $\mK (G_1) \neq \emptyset$. 
For any reaction vector $\bk_1 = (k_{1, \by \to \by'})_{\by \to \by' \in E_1} \in
\mK (G_1)$, we assume the mass-action system $(G_1, \bk_1)$ has a complex-balanced steady state $\bx_1$.
Following \eqref{subtoric_hat}, we obtain the reaction rate vector $\hat{\bk}_1 \in \hat{\mK} (G_1) \subset \RR_{\geq 0}^E$ as follows:
\begin{equation} \label{eq:subtoric_k1_extension}
\hat{\bk}_1 = (\hat{k}_{1, \by \to \by'})_{\by \to \by' \in E}
\ \text{ with } \
\hat{k}_{1, \by \to \by'} :=
\begin{cases}
k_{1, \by \to \by'}, & \text{if } \by \to \by' \in E_1, \\
0, & \text{if } \by \to \by' \notin E_1.
\end{cases}
\end{equation} 
On the other hand, from Lemma \ref{lem:fiber} we get
\begin{equation} \label{eq:subtoric_k*}
\bk_{G}^{*} (\bx_1, \bx^*)
= \left(  \frac{\bk^*_{\by\rightarrow \by'}{(\bx^*)}^{\by}}{(\bx_1)^{\by}} \right)_{\by\rightarrow \by'\in E} \in \mK (G),
\end{equation}
and $\bx_1$ is a complex-balanced steady state for the mass-action system $(G, \bk_{G}^{*} (\bx_1, \bx^*))$.

Now we claim that for any $\varepsilon > 0$, 
$\varepsilon \bk_{G}^{*} (\bx_1, \bx^*) +  \hat{\bk}_1  \in \mK (G)$.
From $\bk_1 \in \mK (G_1)$, we have for every vertex $\by_0 \in V_1 \subset V$
\begin{equation} \notag
\sum_{\by_0 \to \by' \in E_1} k_{1, \by_0 \to \by'} (\bx_1)^{\by_0}
= \sum_{\by \to \by_0 \in E_1} k_{1, \by \to \by_0}
(\bx_1)^{\by}.
\end{equation}
Following \eqref{eq:subtoric_k1_extension}, we derive for every vertex $\by_0 \in V$
\begin{equation} \label{eq:subtoric_1}
\sum_{\by_0 \to \by' \in E} \hat{k}_{1, \by_0 \to \by'} (\bx_1)^{\by_0}
= \sum_{\by \to \by_0 \in E} \hat{k}_{1, \by \to \by_0}
(\bx_1)^{\by}.
\end{equation}
Note that $\bx_1$ is a complex-balanced steady state for the mass-action system $(G, \bk_{G}^{*} (\bx_1, \bx^*))$.
Then for every vertex $\by_0 \in V$
\begin{equation} \notag
\sum_{\by_0 \to \by' \in E} k_{\by_0 \to \by'}^{*} (\bx_1, \bx^*) (\bx_1)^{\by_0} 
= \sum_{\by \to \by_0 \in E} k^*_{\by \to \by_0} (\bx_1, \bx^*)
(\bx_1)^{\by},
\end{equation}
and thus for any $\varepsilon > 0$, 
\begin{equation} \label{eq:subtoric_2}
\sum_{\by_0 \to \by' \in E} \varepsilon k^{*}_{\by_0 \to \by'} (\bx_1, \bx^*) (\bx_1)^{\by_0} 
= \sum_{\by \to \by_0 \in E} \varepsilon k^*_{\by \to \by_0} (\bx_1, \bx^*)
(\bx_1)^{\by}.
\end{equation}
Together with \eqref{eq:subtoric_1} and \eqref{eq:subtoric_2}, we prove the claim.

\smallskip

By passing $\varepsilon \to 0$, we get $\hat{\bk}_1 \in \overline{\mK(G)}$ and $\hat{\mK} (G_1) \subseteq \overline{\mK(G)}$.
It is clear that we can apply this method to all other subgraphs of $G$. Therefore, we conclude
\begin{equation} \notag
\bigcup_{G_i \subseteq G} \hat{\mK} (G_i) \subseteq \overline{\mK(G)}.
\end{equation}

Finally, we show $\bigcup\limits_{G_i \subseteq G} \hat{\mK} (G_i) \neq \overline{\mK(G)} $.
For any vector $\bk \in \mK (G)$ and number $\varepsilon > 0$, we can compute $\varepsilon \bk  \in \mK (G)$.
By passing $\varepsilon \to 0$, we get $\mathbf{0} \in \overline{\mK(G)}$.
However, consider any subgraph $G_i$ of $G$, $\mK (G_i)$ is either an empty set or only contains positive vectors, that is, $\mathbf{0} \notin \mK (G_i)$ or $\hat{\mK} (G_i)$.
\end{proof}

\subsection{Homeomorphisms on the Toric Locus}\label{sec:homeomorphism}

It is difficult to analyze the toric locus due to its nonlinearity. 
Here we introduce the linear flux systems and then establish the \emph{product structure} of the toric locus and its subsets via homeomorphisms involving flux systems.

\begin{definition}

Let $G=(V, E)$ be an E-graph. 
\begin{enumerate}[(a)]
\item For each $\by\rightarrow \by' \in E$, let $J_{\by \rightarrow \by'} >0$ be the associated flux and let $\bJ = (J_{\by \rightarrow \by'})_{\by\rightarrow \by' \in E} \in \mathbb{R}_{>0}^{E}$ be the flux vector. The pair $(G, \bJ)$ is said to be a \emph{flux system}.

\item  A flux vector $\bJ \in \RR_{>0}^E$ is said to be a \emph{complex-balanced flux vector} if for every vertex $\by_0 \in V$,
\begin{equation} \notag
\sum_{\by \to \by_0 \in E} J_{\by \to \by_0} 
= \sum_{\by_0 \to \by' \in E} J_{\by_0 \to \by'}.
\end{equation} 
Then $(G, \bJ)$ is called a \emph{complex-balanced flux system}. We denote the set of all complex-balanced flux vectors on $G$ by
\begin{equation} \notag
\mJ(G):=
\{\bJ \in \RR_{>0}^{E} \mid \bJ  \text{ is a complex-balanced flux vector on G}\}.
\end{equation}
\end{enumerate}
\end{definition}

\begin{definition}

Let $(G,\bJ)$ and $(\tilde{G}, \tilde{\bJ})$ be two flux systems.
\begin{enumerate}[(a)]
\item $(G,\bJ)$ and $(\tilde{G}, \tilde{\bJ})$ are said to be \emph{flux equivalent} (denoted by $(G, \bJ) \sim (\tilde{G}, \tilde{\bJ})$), if for every vertex $\by_0 \in V \cup \tilde{V}$, 
\begin{equation} \notag
\sum_{\by_0 \to \by \in E} J_{\by_0 \to \by} (\by - \by_0) 
= \sum_{\by_0 \to \by' \in \tilde{E}} \tilde{J}_{\by_0 \to \by'} (\by' - \by'_0).
\end{equation}

\item $(\tilde{G}, \tilde{\bJ})$ is said to be \emph{$\RR$-realizable} on $G$ if there exists some $\bJ \in \mathbb{R}^{E}$, such that for every vertex $\by_0 \in V \cup \tilde{V}$,
\begin{equation} \notag
\sum_{\by_0 \to \by \in E} J_{\by \to \by_0} 
(\by' - \by) 
= \sum_{\by_0 \to \by' \in \tilde{E}} \tilde{J}_{\by_0 \to \by'} (\by' - \by'_0).
\end{equation}
Further, if $\bJ \in \mathbb{R}^{E}_{>0}$, $(\tilde{G}, \tilde{\bJ})$ is said to be \emph{realizable} on $G$.

\item We denote the set $\mJe (\tilde{G}, G)$ as
\begin{equation} \notag
\mJe (\tilde{G}, G) := \{  \tilde{\bJ} \in \mJ (\tilde{G}) \ \big| \ \text{the flux system } (\tilde{G},  \tilde{\bJ}) \ \text{is $\RR$-realizable on } G \}.
\end{equation}
Further, we denote the set $\mJ (\tilde{G}, G)$ as
\begin{equation} \notag
\mJ (\tilde{G}, G) := \{  \tilde{\bJ} \in \mJ (\tilde{G}) \ \big| \ \text{the flux system } (\tilde{G},  \tilde{\bJ}) \ \text{is realizable on } G \}.
\end{equation}

\end{enumerate}

\end{definition}

\begin{remark}
Suppose two flux systems $(G,\bJ)$ and $(\tilde{G}, \tilde{\bJ})$ are flux equivalent, then $(G, \bJ)$ is realizable on $\tilde{G}$ and $(\tilde{G}, \tilde{\bJ})$ is realizable on $G$.
\end{remark}

Here we list some of the most important results of flux systems.

\begin{lemma}[\cite{disg_1}]
\label{lem:convex_cone}

Let $\tilde{G} = (\tilde{V}, \tilde{E})$ be a weakly reversible E-graph and let $G = (V, E)$ be an E-graph. Then $\mJ (\tilde{G}, G)$ is a convex cone.
\end{lemma}

\begin{proposition}[\cite{craciun2020efficient}]
\label{prop:craciun2020efficient}
Consider two mass-action systems $(G, \bk)$ and $(\tilde{G}, \tilde{\bk})$. Let $\bx \in \RR_{>0}^n$,
define the flux vector $\bJ(\bx) = (J_{\by \to \by'})_{\by \to \by' \in G} $ on $G$ with $J_{\by \to \by'} = \bk_{\by \to \by'} \bx^{\by} $.
Similarly, define the flux vector $\bJ'(\bx) = (J'_{\by \to \by'})_{\by \to \by' \in G'} $ on $G'$ with $J'_{\by \to \by'} = \bk'_{\by \to \by'} \bx^{\by}$.
Then the following are equivalent:
\begin{enumerate}

\item[(a)] the mass-action systems $(G, \bk)$ and $(\tilde{G}, \tilde{\bk})$ are dynamically equivalent.

\item[(b)] the flux systems $(G, \bJ(\bx))$ and $(\tilde{G}, \tilde{\bJ}(\bx))$ are flux equivalent for all $\bx \in \RR_{>0}^n$.

\item[(c)] the flux systems $(G, \bJ(\bx))$ and $(\tilde{G}, \tilde{\bJ}(\bx))$ are flux equivalent for some $\bx \in \RR_{>0}^n$
\end{enumerate}
\end{proposition}

The following theorem shows the product structure of the toric locus.

\begin{theorem}[\cite{craciun2020structure}]
\label{thm:homeo}
Consider a weakly reversible E-graph $G = (V, E)$ and its stoichiometric subspace $\mSG$.
For any $\bx_0\in\mathbb{R}^n_{>0}$, there exists a homeomorphism
\begin{equation}
\varphi : \mJ(G)\times[(\bx_0 + \mSG)\cap\mathbb{R}^n_{>0}]  \mapsto \mK(G),
\end{equation}
such that for $\bx \in (\bx_0 + \mSG)\cap\mathbb{R}^n_{>0}$ and 
$\bJ = (J_{\byi \to \byj})_{\byi \to \byj \in E} \in \mJ (G)$,
\begin{equation} 
\varphi (\bJ, \bx) = \Big( \frac{J_{\byi \to \byj}}{\bx^{\byi}} \Big)_{\byi \to \byj \in E}.
\end{equation}
Then $\mK(G)$ is 
homeomorphic to the product space $\mJ(G)\times[(\bx_0 + \mSG)\cap\mathbb{R}^n_{>0}]$.
\end{theorem}

Now we are ready to present the main result of this section.

\begin{theorem}
\label{thm:sub_homeo}
Consider a weakly reversible E-graph $\tilde{G} = (\tilde{V}, \tilde{E})$ with its stoichiometric subspace $\mS_{\tilde{G}}$. For any E-graph $G = (V, E)$ and $\bx_0\in\mathbb{R}^n_{>0}$, then
\begin{enumerate}[(a)]
\item $\mKe (\tilde{G}, G)$ is homeomorphic to the product space $\mJe (\tilde{G}, G) \times [(\bx_0 + \mS_{\tilde{G}})\cap\mathbb{R}^n_{>0}]$.

\item $\mK(\tilde{G}, G)$ is homeomorphic to the product space $\mJ(\tilde{G}, G)\times[(\bx_0 + \mS_{\tilde{G}})\cap\mathbb{R}^n_{>0}]$.
\end{enumerate}

\end{theorem}

\begin{proof}

We only prove part $(b)$ as part $(a)$ can be verified similarly. From Theorem \ref{thm:homeo}, there exists the homeomorphism map $\varphi$, such that 
\begin{equation} \notag
\varphi : \mJ (\tilde{G}) \times [(\bx_0 + \mS_{\tilde{G}} )\cap\mathbb{R}^n_{>0}]  \mapsto \mK (\tilde{G}).
\end{equation}


On the subdomain $\mJ(\tilde{G}, G) \times [(\bx_0+\mS_{\tilde{G}})\cap\mathbb{R}^n_{>0}]$, the injectivity and continuity of $\varphi$ and its inverse are guaranteed from Theorem~\ref{thm:homeo}. Since $\mK(\tilde{G}, G) \subset \mK (\tilde{G})$ and $\mJ(\tilde{G}, G) \subset \mJ (\tilde{G})$, to prove the homeomorphism in part $(b)$, it therefore suffices for us to show that
\begin{equation} \label{varphi_sub}
\varphi \big( \mJ(\tilde{G}, G) \times [(\bx_0+\mS_{\tilde{G}})\cap\mathbb{R}^n_{>0}] \big) = \mK(\tilde{G}, G).
\end{equation}

First, we show that $\varphi \big( \mJ(\tilde{G}, G) \times [(\bx_0+\mS_{\tilde{G}})\cap\mathbb{R}^n_{>0}] \big) \subseteq \mK(\tilde{G}, G)$. Assume that the pair 
$(\tilde{\bJ}, \tilde{\bx}) \in \mJ(\tilde{G}, G) \times [(\bx_0+\mS_{\tilde{G}})\cap\mathbb{R}^n_{>0}]$, 
there exists a flux vector $\bJ = (J_{\byi \to \byj})_{\byi \to \byj \in E} \in \mathbb{R}^{E}_{>0}$, such that for every vertex $\by_0 \in V \cup \tilde{V}$,
\begin{equation} \label{eq:sub_homeo_1}
\sum_{\by_0 \to \by \in E} J_{\by_0 \to \by} 
(\by - \by_0) 
= \sum_{\by_0 \to \by' \in \tilde{E}} \tilde{J}_{\by_0 \to \by'} (\by' - \by_0).
\end{equation}
Using Theorem \ref{thm:homeo}, we input $\varphi (\tilde{\bJ}, \tilde{\bx})$
into \eqref{eq:sub_homeo_1}. Then for every vertex $\by_0 \in V \cup \tilde{V}$,
\begin{equation} \notag
\sum_{\by_0 \to \by \in E} J_{\by_0 \to \by} 
(\by - \by_0) 
= \sum_{\by_0 \to \by' \in \tilde{E}} \varphi_{\by_0 \to \by'}  (\tilde{\bx})^{\by_0} (\by' - \by_0).
\end{equation}
We construct the following reaction rate vector in $G$: 
\begin{equation} \notag
\bk = (k_{\byi \to \byj})_{\byi \to \byj \in E} \ \text{ with } \
k_{\byi \to \byj} = \frac{J_{\byi \to \byj}}{(\tilde{\bx})^{\byi}},
\end{equation} 
and for every vertex $\by_0 \in V \cup \tilde{V}$
\begin{equation} \label{eq:sub_homeo_3}
\sum_{\by_0 \to \by \in E} k_{\by_0 \to \by} (\tilde{\bx})^{\by_0}
(\by - \by_0) 
= \sum_{\by_0 \to \by' \in \tilde{E}} \varphi_{\by_0 \to \by'}  (\tilde{\bx})^{\by_0} (\by' - \by_0).
\end{equation}
From Proposition \ref{prop:craciun2020efficient}, we derive 
\[
(\tilde{G}, \varphi (\tilde{\bJ}, \tilde{\bx})) \sim (G, \bk)
\ \text{ and } \
\varphi (\tilde{\bJ}, \tilde{\bx}) \in \mK(\tilde{G}, G).
\]
Thus, we prove that $\varphi \big( \mJ(\tilde{G}, G) \times [(\bx_0+\mS_{\tilde{G}})\cap\mathbb{R}^n_{>0}] \big) \subseteq \mK(\tilde{G}, G)$.

\smallskip

Second, we show $\mK(\tilde{G}, G)\subseteq \varphi \big( \mJ(\tilde{G}, G) \times [(\bx_0+\mS_{\tilde{G}})\cap\mathbb{R}^n_{>0}] \big)$. Assume $\tilde{\bk} \in \mK(\tilde{G}, G)$,
there exists a reaction rate vector $\bk = (k_{\byi \to \byj})_{\byi \to \byj \in E} \in \mathbb{R}^{E}_{>0}$, such that for every vertex $\by_0 \in V \cup \tilde{V}$
\begin{equation} \label{eq:sub_homeo_4}
\sum_{\by_0 \to \by \in E} k_{\by_0 \to \by} 
(\by - \by_0) 
= \sum_{\by_0 \to \by' \in \tilde{E}} \tilde{k}_{\by_0 \to \by'} (\by' - \by_0).
\end{equation}
Note that $(\tilde{G}, \tilde{\bk})$ has a unique complex-balanced steady state in $(\bx_0+\mS_{\tilde{G}})\cap\mathbb{R}^n_{>0}$, denoted by $\tilde{\bx}$.
Then we build two flux vectors in $\tilde{G}$ and $G$ as follows:
\begin{equation} \notag
\begin{split}
& \tilde{\bJ} = (\tilde{J}_{\byi \to \byj})_{\byi \to \byj \in \tilde{E}}
\ \text{ with } \
\tilde{J}_{\byi \to \byj} = \tilde{k}_{\byi \to \byj}(\tilde{\bx})^{\byi},
\\& \bJ = (J_{\byi \to \byj})_{\byi \to \byj \in E}
\ \text{ with } \
J_{\byi \to \byj} = k_{\byi \to \byj}(\tilde{\bx})^{\byi}.
\end{split}
\end{equation}
Thus, we get
\begin{equation} \label{eq:sub_homeo_6}
\sum_{\by_0 \to \by \in E} J_{\by_0 \to \by}
(\by - \by_0) 
= \sum_{\by_0 \to \by' \in \tilde{E}} \tilde{J}_{\by_0 \to \by'} (\by' - \by_0).
\end{equation}
This implies that $\tilde{\bJ} \in \mJ(\tilde{G}, G)$ and $\varphi (\tilde{\bJ}, \tilde{\bx}) = \tilde{\bk}$. 
Therefore, together with two parts, 
we conclude~\eqref{varphi_sub}.
\end{proof}

From the homeomorphism in Theorem \ref{thm:sub_homeo}, we show the connectivity on $\mK(\tilde{G}, G)$.

\begin{theorem}
\label{thm:connect}
Consider a weakly reversible E-graph $\tilde{G} = (\tilde{V}, \tilde{E})$ and an E-graph $G = (V, E)$. 
\begin{enumerate}[(a)]
\item Both $\mJe (\tilde{G}, G)$ and $\mKe (\tilde{G}, G)$ are 
path-connected.
Moreover, $\mJe (\tilde{G}, G)$ is non-empty if and only if $\mKe (\tilde{G}, G)$ is non-empty.

\item Both $\mJ(\tilde{G}, G)$ and $\mK(\tilde{G}, G)$ are 
path-connected.
Moreover, $\mJ(\tilde{G}, G)$ is non-empty if and only if $\mK(\tilde{G}, G)$ is non-empty.
\end{enumerate}
\end{theorem}

\begin{proof}

We only prove part $(b)$ since part $(a)$ follows analogously. First, suppose $\mJ(\tilde{G}, G) = \emptyset$. From Theorem \ref{thm:sub_homeo}, $\mK(\tilde{G}, G)$ is homeomorphic to the product space $\mJ(\tilde{G}, G)\times[(\bx_0 + \mS_{\tilde{G}})\cap\mathbb{R}^n_{>0}]$. Thus we derive that $\mK(\tilde{G}, G) = \emptyset$. 

\smallskip

Next, suppose $\mJ(\tilde{G}, G) \neq \emptyset$.
From Definition and Lemma \ref{lem:convex_cone}, $\mJ(\tilde{G}, G)$ and $(\bx_0 + \mS_{\tilde{G}}) \cap \mathbb{R}^n_{>0}$ are both path-connected.
Using Theorem \ref{thm:sub_homeo}, $\mK(\tilde{G}, G)$ is homeomorphic to the product space $\mJ(\tilde{G}, G)\times[(\bx_0 + \mS_{\tilde{G}})\cap\mathbb{R}^n_{>0}]$, we conclude $\mK(\tilde{G}, G) \neq \emptyset$ and $\mK(\tilde{G}, G)$ is path-connected.
\end{proof}

The following is another consequence of Theorem \ref{thm:homeo} and Lemma \ref{lem:fiber}.

\begin{lemma}
Consider a weakly reversible E-graph $G = (V, E)$ and its stoichiometric subspace $\mSG$. For any $\bk^* \in \mK(G)$ and $\bx_0 \in \mathbb{R}^n_{>0}$, the mass-action system $(G, \bk^*)$ has a unique steady state $\bx^* \in (\bx_0 + S_{G})\cap\mathbb{R}^n_{>0}$.
Define
\begin{equation}
\mathcal{K}_{G} (\bk^*, \bx^*) := \{ \bk_{G}^{*} (\bx, \bx^*) \ | \ \bx \in (\bx_0 + \mSG )\cap\mathbb{R}^n_{>0} \},
\end{equation}
where
\begin{equation} \notag
\bk_{G}^{*} (\bx, \bx^*)
= \left(  \frac{k^*_{\by\rightarrow \by'}{(\bx^*)}^{\by}}{\bx^{\by}} \right)_{\by\rightarrow \by'\in E}.
\end{equation}
Then $\mathcal{K}_{G} (\bk^*, \bx^*)$ is homeomorphic to the stoichiometric compatibility class $(\bx_0 + S_{G})\cap\mathbb{R}^n_{>0}$.
\end{lemma}

\begin{proof}
From Theorem \ref{thm:homeo}, there exists the homeomorphism map $\varphi$, such that 
\begin{equation} \notag
\varphi : \mJ(G) \times [(\bx_0 + \mS_{G} )\cap\mathbb{R}^n_{>0}]  \mapsto \mK (G).
\end{equation}
Using Lemma \ref{lem:fiber}, we have for any $\bx \in (\bx_0 + \mS_{G} )\cap\mathbb{R}^n_{>0}$
\begin{equation} \notag
\bk_{G}^{*} (\bx, \bx^*) \in \mK (G),
\end{equation}
and $\bx$ is the unique steady state for the mass-action system 
$(G, \bk_{G}^{*} (\bx, \bx^*))$ in $(\bx_0 + \mS_{G} )\cap\mathbb{R}^n_{>0}$.

Now let $\bJ = ( \bk^*_{\by\rightarrow \by'}{(\bx^*)}^{\by} )_{\by\rightarrow \by'\in E}$, we can check that for any $\bx \in (\bx_0 + \mS_{G} )\cap\mathbb{R}^n_{>0}$,
\begin{equation}
\varphi (\bJ, \bx) = \bk_{G}^{*} (\bx, \bx^*).
\end{equation}
Thus, it suffices for us to prove 
\begin{equation} \label{varphi_sub_2}
\varphi \big( \{ \bJ \} \times [(\bx_0+\mS_{G})\cap\mathbb{R}^n_{>0}] \big) = \mathcal{K}_{G} (\bk^*, \bx^*).
\end{equation}
The rest of the proof follows directly from the steps in the proof of Theorem~\ref{thm:sub_homeo}. 
\end{proof}




\begin{example}
\label{ex:1}

Consider two E-graphs in Figure~\ref{fig:disguised_toric_1}, we give an example of an application of Theorem~\ref{thm:connect}.
The vertices shown in both E-graphs are given by
\begin{equation}  
\by_1 = \begin{pmatrix} 0 \\ 3 \end{pmatrix}, \quad 
\by_2 = \begin{pmatrix} 2 \\ 1 \end{pmatrix}, \quad 
\by_3 = \begin{pmatrix} 1 \\ 2 \end{pmatrix}, \quad
\by_4 = \begin{pmatrix} 3 \\ 0 \end{pmatrix}.
\end{equation}

\begin{figure}[!ht]
\centering
\includegraphics[scale=0.4]{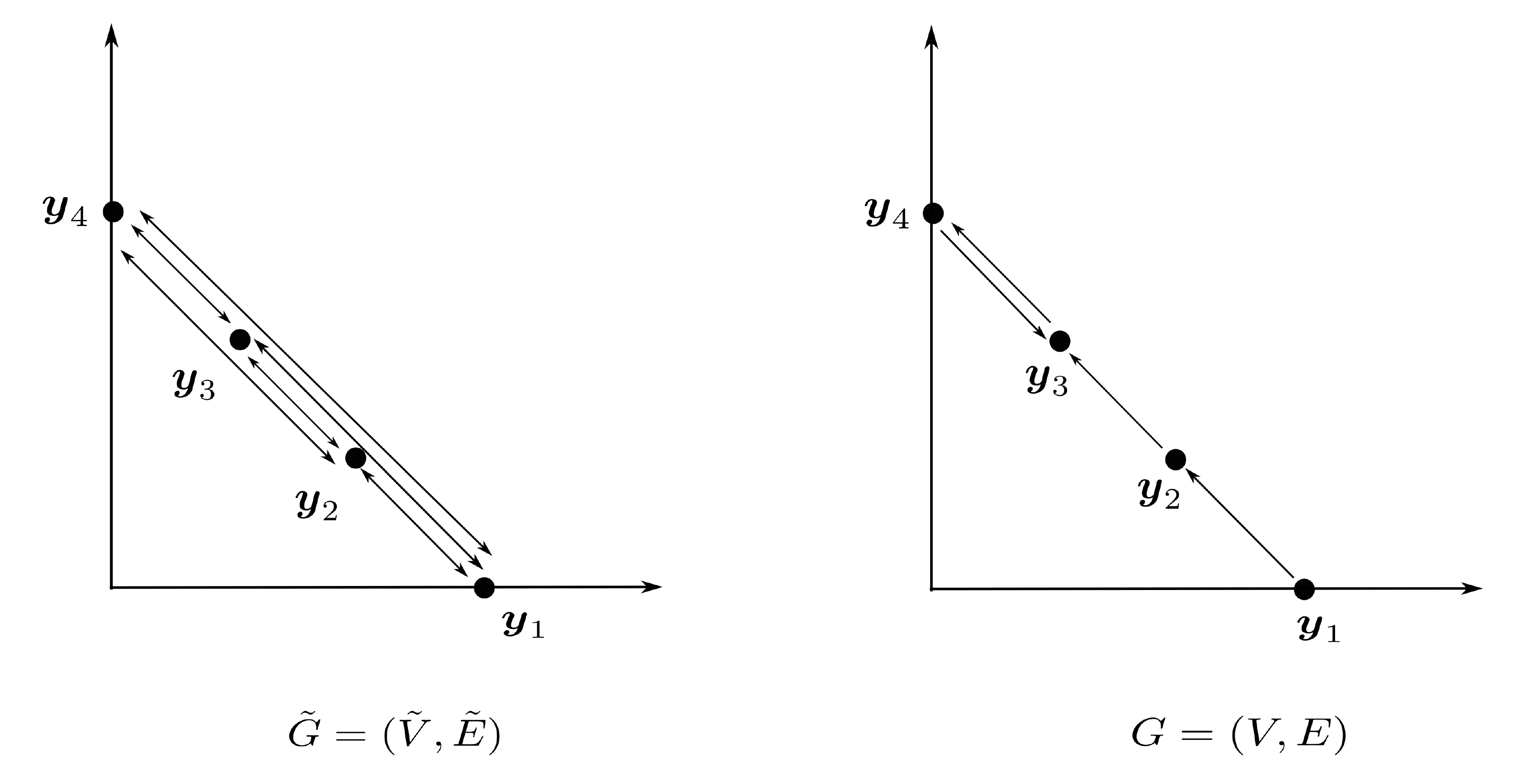}
\caption{\small Two E-graphs $\tilde{G} =(\tilde{V}, \tilde{E})$ and $G =(V, E)$. Note that $G \subseteq \tilde{G}$.}
\label{fig:disguised_toric_1}
\end{figure} 

For the weakly reversible E-graph $\tilde{G} = (\tilde{V}, \tilde{E})$, from Theorem \ref{thm:homeo} and Lemma \ref{lem:convex_cone} we get the path-connectedness on $\mK (\tilde{G})$.
Now consider the E-graph $G = (V, E)$. 

\begin{enumerate}[(a)]

\item First, we claim that
\[
\mKe (\tilde{G}, G) = \mK (\tilde{G}).
\]
For any $\tilde{\bk} \in \mK (\tilde{G})$, we set $\bk \in \RR^{|E|}$ as
\begin{equation}
\begin{split}
& \bk_{12} = \tilde{\bk}_{12} + 2 \tilde{\bk}_{13} + 3 \tilde{\bk}_{14},
\ \ \bk_{23} = \tilde{\bk}_{23} - \tilde{\bk}_{21} + 2 \tilde{\bk}_{24}.
\\& \bk_{34} = \tilde{\bk}_{34} - \tilde{\bk}_{32} - 2 \tilde{\bk}_{31},
\ \ \ \bk_{43} = \tilde{\bk}_{43} + 2 \tilde{\bk}_{42} + 3 \tilde{\bk}_{41}.
\end{split}
\end{equation}
From Definition \ref{def:de}, we can compute that $(\tilde{G}, \tilde{\bk}) \sim (G, \bk)$ and thus prove the claim.
Therefore, we conclude the path-connectedness on $\mKe (\tilde{G}, G)$.

\item Second, we consider $\mK (\tilde{G}, G)$.
For any $\tilde{\bk} \in \mK (\tilde{G})$, to ensure $(\tilde{G}, \tilde{\bk}) \sim (G, \bk)$ we set 
\begin{equation}
\begin{split}
& \bk_{12} = \tilde{\bk}_{12} + 2 \tilde{\bk}_{13} + 3 \tilde{\bk}_{14},
\ \ \bk_{23} = \tilde{\bk}_{23} - \tilde{\bk}_{21} + 2 \tilde{\bk}_{24}.
\\& \bk_{34} = \tilde{\bk}_{34} - \tilde{\bk}_{32} - 2 \tilde{\bk}_{31},
\ \ \ \bk_{43} = \tilde{\bk}_{43} + 2 \tilde{\bk}_{42} + 3 \tilde{\bk}_{41}.
\end{split}
\end{equation}
Moreover, in this case we need $\bk \in \RR^{|E|}_{>0}$, that is, 
\[
\tilde{\bk}_{23} + 2 \tilde{\bk}_{24} \geq \tilde{\bk}_{21}, \ \ 
\tilde{\bk}_{34} \geq \tilde{\bk}_{32} + 2 \tilde{\bk}_{31}.
\]
Thus we obtain that
\[
\mK (\tilde{G}, G) = \{ \tilde{\bk} \in \mK (\tilde{G}) \ \big| \ \tilde{\bk}_{23} + 2 \tilde{\bk}_{24} \geq \tilde{\bk}_{21} 
\ \text{ and } \
\tilde{\bk}_{34} \geq \tilde{\bk}_{32} + 2 \tilde{\bk}_{31} \}.
\]
Applying Theorem~\ref{thm:connect}, we conclude the path-connectedness on $\mK (\tilde{G}, G)$.
\end{enumerate}
\qed
\end{example}

\section{The Disguised Toric Locus is Path-Connected}
\label{sec:main_result}

In this section, we show the main results of this paper: both the disguised toric locus and the  $\RR$-disguised toric locus are \emph{path-connected}.

\begin{theorem}
\label{thm:dis_toric}

Consider an E-graph $G=(V,E)$. Then the $\RR$-disguised toric locus of $G$ is 
path-connected.\footnote{Some authors exclude the empty set from being path-connected, but we do not follow this convention here.\label{empty}}
\end{theorem}

\begin{proof}

If the $\RR$-disguised toric locus of $G$ is an empty set, it is path-connected. Else, we proceed as follows:

\textbf{Step 1:}
Recall that for any weakly reversible subgraph $G_i$ of $ G_{\rm{comp}}$, the $\RR$-disguised toric locus of $G$ is given by
\begin{equation} \notag
\dK (G) = \bigcup_{G_i \sqsubseteq G_{\rm{comp}}} \ \dK(G, G_i).
\end{equation}

If $\dK (G) = \emptyset$, then we prove the theorem. Otherwise, suppose a reaction rate vector $\bk \in \dK (G)$. Let $\bk \in \dK(G, G_1)$ with $G_1$ is a weakly reversible subgraph of $ G_{\rm{comp}}$, and there exists a reaction rate vector $\bk_1 \in \mK (G_1)$, such that
$(G, \bk) \sim (G_1, \bk_1)$.

Consider a fixed state $\bx_0\in\mathbb{R}^n_{>0}$, the mass-action system $(G_1, \bk_1)$ has a unique steady state $\bx_1 \in (\bx_0 + S_{G})\cap\mathbb{R}^n_{>0}$.
Then we define the following set of reaction vectors:
\begin{equation} \notag
\mathcal{K}_{(G,\bk)} (\bx_1) := \{ \bk_{G} (\bx, \bx_1) \in \RR^E \ | \ \bx \in (\bx_0 + \mSG )\cap\mathbb{R}^n_{>0} \},
\end{equation}
where
\begin{equation} \notag
 \bk_{G} (\bx, \bx_1)
= (k_{\by\rightarrow \by'} (\bx, \bx_1))_{\by\rightarrow \by'\in E}
= \left(  \frac{k_{\by\rightarrow \by'}{(\bx_1)}^{\by}}{\bx^{\by}} \right)_{\by\rightarrow \by'\in E}.
\end{equation}

Here we claim that $\mathcal{K}_{(G,\bk)} (\bx_1) \subseteq \dK(G, G_1)$.
It suffices to show for any $\bx \in (\bx_0 + \mSG )\cap \mathbb{R}^n_{>0}$,
$\bk_{G} (\bx, \bx_1) \in \dK(G, G_1)$.
Since $(G, \bk) \sim (G_1, \bk_1)$, 
for every vertex $\by_0 \in V \cup V_1$
\begin{equation} \notag
\sum_{\by_0 \to \by \in E} k_{\by_0  \to \by} (\by - \by_0) 
= \sum_{\by_0 \to \by' \in E_1} k_{1, \by_0  \to \by'}  (\by' - \by_0).
\end{equation}
Since $\bx \in (\bx_0 + \mSG )\cap\mathbb{R}^n_{>0}$, for every vertex $\by_0 \in V \cup V_1$
\begin{equation} \notag
\sum_{\by_0 \to \by \in E} k_{\by_0  \to \by} 
 \frac{(\bx_1)^{\by} }{\bx^{\by}} (\by - \by_0) 
= \sum_{\by_0 \to \by' \in E_1} k_{1, \by_0  \to \by'} \frac{(\bx_1)^{\by} }{\bx^{\by}} (\by' - \by_0).
\end{equation}
Thus, we derive that
\begin{equation} \label{eq:k(x)}
(G, \bk_{G} (\bx, \bx_1)) \sim (G_1, \bk_1 (\bx) )
\ \text{ with } \
k_{1, \by\rightarrow \by'} (\bx) = \frac{k_{1, \by\rightarrow \by'}{(\bx_1)}^{\by}}{\bx^{\by}}.
\end{equation} 
Using Lemma \ref{lem:fiber}, we have 
\begin{equation} \label{eq:k(x)K(G_1)}
\bk_1 (\bx) \in \mK (G_1)
\ \text{ and } \
\bx \text{ is a complex-balanced steady state for }
(G_1, \bk_1 (\bx) )
\end{equation}
Therefore, $\bk_{G} (\bx, \bx_1) \in \dK(G, G_1)$ and we prove the claim.

\medskip

\textbf{Step 2:}
Suppose two reaction rate vectors $\bk^a, \bk^b \in \dK (G)$. Let $\bk^a \in \dK(G, G_1)$ and $\bk^b \in \dK(G, G_2)$, where $G_1$ and $G_2$ are weakly reversible subgraphs of $ G_{\rm{comp}}$.
There exist two corresponding reaction rate vectors $\bk_1 \in \mK (G_1)$ and $\bk_2 \in \mK (G_2)$, such that 
\begin{equation} \label{eq:dis_toric_0}
(G, \bk^a) \sim (G_1, \bk_1)
\ \text{ and } \
(G, \bk^b) \sim (G_2, \bk_2).
\end{equation}
Further, we suppose $(G_1, \bk_1)$ and $(G_2, \bk_2)$ share one complex-balanced steady state $\bx_0\in\mathbb{R}^n_{>0}$. Now we claim that
\begin{equation} \notag
L (\bk^a, \bk^b) := \{ \alpha \bk^a + (1 - \alpha) \bk^b \ | \ 0 \leq \alpha \leq 1 \}
\subset \dK (G).
\end{equation}
Note that when $\alpha = 0, 1$, we have $\bk^a, \bk^b \in \dK (G)$.

From Definition \ref{def:cb_system} and $\bx_0$ is a complex-balanced steady state for $(G_1, \bk_1)$ and $(G_2, \bk_2)$, we have for every vertex $\by_0 \in V_1 \cup V_2$
\begin{equation} 
\begin{split} \label{eq:dis_toric_1}
& \sum_{\by_0 \to \by' \in E} k_{1,\by_0 \to \by'} (\bx_0)^{\by_0}
= \sum_{\by \to \by_0 \in E} k_{1, \by \to \by_0}
(\bx_0)^{\by}, 
\\& \sum_{\by_0 \to \by' \in E} k_{2, \by_0 \to \by'} (\bx_0)^{\by_0}
= \sum_{\by \to \by_0 \in E} k_{2, \by \to \by_0}
(\bx_0)^{\by}.
\end{split}
\end{equation}
Define an E-graph $G_{1,2} = (V_{1,2}, E_{1,2})$ with $V_{1,2} := V_1 \cup V_2$ and $E_{1,2} := E_1 \cup E_2$.
It is clear $G_{1,2}$ is a weakly reversible subgraph of $G_{\rm{comp}}$.
Given a fixed number $0 < \alpha < 1$, from \eqref{eq:dis_toric_1} we derive that for every vertex $\by_0 \in V_{1,2}$,
\begin{equation} \notag
\sum_{\by_0 \to \by' \in E_{1,2}} \left( \alpha k_{1, \by_0 \to \by'} + (1-\alpha)  k_{2, \by_0 \to \by'} \right) (\bx_0)^{\by_0}
= \sum_{\by \to \by_0 \in E_{1,2}} \left( \alpha k_{1, \by \to \by_0} + (1-\alpha)  k_{2, \by \to \by_0} \right)
(\bx_0)^{\by}.
\end{equation}
This shows that $\alpha \bk_1 + (1-\alpha) \bk_2 \in \mK (G_{1, 2})$.
On the other hand, from \eqref{eq:dis_toric_0} we can check 
\begin{equation} \notag
(G, \alpha \bk^a + (1-\alpha) \bk^b) \sim (G_{1,2}, \alpha \bk_1 + (1-\alpha) \bk_2),
\end{equation}
and thus conclude $\alpha \bk^a + (1-\alpha) \bk^b \in \dK (G)$.

\medskip

\textbf{Step 3:}
Now we show $\dK (G)$ is path-connected.
Consider any two reaction rate vectors in $\dK (G)$ such that
\[
\bk^a \in \dK(G, G_1)
\ \text{ and } \
\bk^b \in \dK(G, G_2),
\]
where
$G_1$ and $G_2$ are weakly reversible subgraphs of $ G_{\rm{comp}}$.
Let $\bk_1 \in \mK (G_1)$ and $\bk_2 \in \mK (G_2)$ such that $(G, \bk^a) \sim (G_1, \bk_1)$ and $(G, \bk^b) \sim (G_2, \bk_2)$.
For a fixed state $\bx_0\in\mathbb{R}^n_{>0}$, the mass-action systems $(G_1, \bk_1)$ and $(G_2, \bk_2)$ have steady states $\bx_1$ and $\bx_2$ in
$(\bx_0 + S_{G})\cap\mathbb{R}^n_{>0}$, respectively.

In step 1, we construct $\mathcal{K}_{(G,\bk^a)} (\bx_1)$ and $\mathcal{K}_{(G,\bk^b)} (\bx_2)$. It is clear that both of them are path-connected. From $\bx_0\in\mathbb{R}^n_{>0}$, \eqref{eq:k(x)} and \eqref{eq:k(x)K(G_1)}, we get
\begin{equation} \notag
(G_1, \bk_1 (\bx_0) ) \sim  (G, \bk^a_{G} (\bx_0, \bx_1)) 
\ \text{ and } \
(G_2, \bk_2 (\bx_0) ) \sim (G, \bk^b_{G} (\bx_0, \bx_2)),
\end{equation}
and $\bx_0$ is a steady state for both $(G_1, \bk_1 (\bx_0) )$ and $(G_2, \bk_2 (\bx_0) )$.
Further, in step 2 we prove 
\begin{equation*}
L (\bk^a_{G} (\bx, \bx_0), \bk^b_{G} (\bx, \bx_0)) \subset \dK (G).
\end{equation*}
Therefore, there exists a path connecting $\bk^a$ and $\bk^b$, and we prove this theorem.
\end{proof}

Using a similar argument as in Theorem \ref{thm:dis_toric}, we conclude the following remark:

\begin{remark}
Consider an E-graph $G = (V, E)$ and any weakly reversible subgraph $\tilde{G} \sqsubseteq G_{\rm{comp}}$. Then $\dK(G, \tilde{G})$ is 
path-connected.
\end{remark}

The proof of the connectivity of the disguised toric locus is similar to the case of the $\RR$-disguised toric locus. We sketch the proof for completeness below.


\begin{theorem}
\label{thm:ex_dis_toric}

Consider an E-graph $G=(V,E)$. Then the disguised toric locus of $G$ is 
path-connected.\footref{empty}
\end{theorem}

\begin{proof}

If the disguised toric locus of $G$ is an empty set, it is path-connected. Else, we proceed in a way that is very similar to the proof of the previous theorem, as follows.

\smallskip

For any weakly reversible subgraph $G_i$ of $ G_{\rm{comp}}$, the disguised toric locus of $G$ is defined as
\begin{equation} \notag
\mKd (G) = \bigcup_{G_i \sqsubseteq G_{\rm{comp}}} \ \mK_{\rm{disg}}(G, G_i).
\end{equation}
The Theorem immediately follows if $\mKd (G) = \emptyset$. Otherwise, suppose a reaction rate vector $\bk \in \mKd (G) \subset \RR^E_{>0}$, and $\bk \in \mKd(G, G_1)$ with $G_1$ is a weakly reversible subgraph of $ G_{\rm{comp}}$. Then there exists a reaction rate vector $\bk_1 \in \mK (G_1)$, such that
$(G, \bk) \sim (G_1, \bk_1)$.

Consider a fixed state $\bx_0\in\mathbb{R}^n_{>0}$, the mass-action system $(G_1, \bk_1)$ has a unique steady state $\bx_1 \in (\bx_0 + S_{G})\cap\mathbb{R}^n_{>0}$.
Then we define the following set of reaction vectors:
\begin{equation} \notag
\mathcal{K}_{(G,\bk)} (\bx_1) := \{ \bk_{G} (\bx, \bx_1) \in \RR^E_{>0} \ | \ \bx \in (\bx_0 + \mSG )\cap\mathbb{R}^n_{>0} \},
\end{equation}
where
\begin{equation} \notag
 \bk_{G} (\bx, \bx_1)
= (k_{\by\rightarrow \by'} (\bx, \bx_1))_{\by\rightarrow \by'\in E}
= \left(  \frac{k_{\by\rightarrow \by'}{(\bx_1)}^{\by}}{\bx^{\by}} \right)_{\by\rightarrow \by'\in E}.
\end{equation}
Using a similar argument as in Theorem \ref{thm:dis_toric}, we conclude that $\mathcal{K}_{(G,\bk)} (\bx_1) \subseteq \mKd(G, G_1)$.

Suppose two reaction rate vectors $\bk^a, \bk^b \in \mKd (G)$. Let $G_1, G_2$ be two weakly reversible subgraphs of $ G_{\rm{comp}}$ and let $\bk^a \in \mKd(G, G_1)$ and $\bk^b \in \mKd(G, G_2)$.
There exist two corresponding reaction rate vectors $\bk_1 \in \mK (G_1)$ and $\bk_2 \in \mK (G_2)$, such that 
\begin{equation} \notag
(G, \bk^a) \sim (G_1, \bk_1)
\ \text{ and } \
(G, \bk^b) \sim (G_2, \bk_2).
\end{equation}
Further, we suppose $(G_1, \bk_1)$ and $(G_2, \bk_2)$ share one complex-balanced steady state $\bx_0\in\mathbb{R}^n_{>0}$. 
Again using a similar argument as in Theorem \ref{thm:dis_toric}, we get that
\begin{equation} \notag
L (\bk^a, \bk^b) := \{ \alpha \bk^a + (1 - \alpha) \bk^b \ | \ 0 \leq \alpha \leq 1 \}
\subset \mKd (G).
\end{equation}
Note that when $\alpha = 0, 1$, we have $\bk^a, \bk^b \in \mKd (G)$.

Finally, we show $\mKd (G)$ is path-connected.
Consider any two reaction rate vectors in $\mKd (G)$ such that
\[
\bk^a \in \mKd(G, G_1)
\ \text{ and } \
\bk^b \in \mKd(G, G_2),
\]
where $G_1$ and $G_2$ are weakly reversible subgraphs of $ G_{\rm{comp}}$.
Let $\bk_1 \in \mK (G_1)$ and $\bk_2 \in \mK (G_2)$ such that $(G, \bk^a) \sim (G_1, \bk_1)$ and $(G, \bk^b) \sim (G_2, \bk_2)$.
For a fixed state $\bx_0\in\mathbb{R}^n_{>0}$, the mass-action systems $(G_1, \bk_1)$ and $(G_2, \bk_2)$ have steady states $\bx_1$ and $\bx_2$ in
$(\bx_0 + S_{G})\cap\mathbb{R}^n_{>0}$, respectively.

Then we construct $\mathcal{K}_{(G,\bk^a)} (\bx_1)$ and $\mathcal{K}_{(G,\bk^b)} (\bx_2)$. 
From the steps above, both of them are path-connected and we have
\begin{equation} \notag
(G_1, \bk_1 (\bx_0) ) \sim  (G, \bk^a_{G} (\bx_0, \bx_1)) 
\ \text{ and } \
(G_2, \bk_2 (\bx_0) ) \sim (G, \bk^b_{G} (\bx_0, \bx_2)),
\end{equation}
and $\bx_0$ is a steady state for both $(G_1, \bk_1 (\bx_0) )$ and $(G_2, \bk_2 (\bx_0) )$.
We also recall that 
\begin{equation*}
L (\bk^a_{G} (\bx, \bx_0), \bk^b_{G} (\bx, \bx_0)) \subset \mKd (G).
\end{equation*}
Therefore, there exists a path connecting $\bk^a$ and $\bk^b$, and we prove this theorem.
\end{proof}


\begin{example}
\label{ex:2}

Revisit two E-graphs $\tilde{G} =(\tilde{V}, \tilde{E})$ and $G =(V, E)$ in Figure~\ref{fig:disguised_toric_1}. Now we consider the disguised toric locus $\mKd (G, \tilde{G})$ and the $\RR$-disguised toric locus $\dK (G, \tilde{G})$. Recall from Example \ref{ex:1},
suppose that $\tilde{\bk} \in \mK (\tilde{G})$, for any $\bk \in \mKd (G, \tilde{G})$ or $\dK (G, \tilde{G})$, we first need to ensure $(\tilde{G}, \tilde{\bk}) \sim (G, \bk)$. Under direct computation, we have
\begin{equation}
\begin{split}
& \bk_{12} = \tilde{\bk}_{12} + 2 \tilde{\bk}_{13} + 3 \tilde{\bk}_{14},
\ \ \bk_{23} = \tilde{\bk}_{23} - \tilde{\bk}_{21} + 2 \tilde{\bk}_{24}.
\\& \bk_{34} = \tilde{\bk}_{34} - \tilde{\bk}_{32} - 2 \tilde{\bk}_{31},
\ \ \ \bk_{43} = \tilde{\bk}_{43} + 2 \tilde{\bk}_{42} + 3 \tilde{\bk}_{41}.
\end{split}
\end{equation}


\begin{enumerate}[(a)]

\item From~\cite[Theorem 4.3]{2022disguised}, we get that $\mKd (G, \tilde{G}) = \RR^{|E|}_{>0}$.
Thus for any $\bk \in \RR^{|E|}_{>0}$, we have
\[
(G, \bk) \sim (\tilde{G}, \tilde{\bk}),
\text{ for some } \tilde{\bk} \in \mK (\tilde{G}).
\]
It is clear that $\RR^{|E|}_{>0}$ is path-connected, and we conclude the path-connectedness on $\mKd (G, \tilde{G})$.

\item From~\cite[Theorem 4.3]{2022disguised}, for any $\bk \in \dK (G, \tilde{G})$, it needs to satisfy 
\begin{equation}
\bk_{12} > 0, \ \ \bk_{43} > 0.
\end{equation}
Further, if $\bk_{34} > 0 > \bk_{23}$, then it also needs that
\begin{equation}
\bk_{12} \bk_{43} + \bk_{34} \bk_{23} \geq 0.
\end{equation}

It is clear that the region restricted by $\bk_{12}, \bk_{43} > 0$ and
\begin{equation} \notag
\begin{split}
\bk_{23}, \bk_{34} > 0
\ \text{ or } \
\bk_{23}, \bk_{34} < 0
\ \text{ or } \
\bk_{23} > 0 > \bk_{34},
\end{split}   
\end{equation}
is path-connected.
On the other hand, for any $\bk_{34} > 0 > \bk_{23}$, we set
\[
\bk_{12} = - \bk_{34} \bk_{23} / \bk_{43} > 0.
\]
This implies $\dK (G, \tilde{G})$ is path-connected when $\bk_{12}, \bk_{43}, \bk_{34} > 0$ and $\bk_{23} < 0$.
Therefore, we conclude the path-connectedness on $\dK (G, \tilde{G})$.
\end{enumerate}
\qed
\end{example}

\section{Discussion}
\label{sec:discussion}

The notions of \emph{toric locus} and \emph{disguised toric locus} of a reaction network $\mathcal N$ have been studied in depth in recent work~\cite{CraciunDickensteinShiuSturmfels2009, craciun2020structure,2022disguised}, due to the fact that they determine sets in the parameter space of $\mathcal N$ where the dynamics is guaranteed to be remarkably stable~\cite{anderson2011proof, craciun2013persistence, craciun2020efficient, feinberg2019foundations, yu2018mathematical}. The \emph{toric locus} is the set of parameter values of $\mathcal N$ that generate {\em complex-balanced} dynamical systems; while the \emph{disguised toric locus} is a larger set in the parameter space of $\mathcal N$, which generate dynamical systems that are {\em realizable} as complex-balanced systems, possibly by using another network $\mathcal N'$. In particular, it has been shown that the toric locus is connected and that steady states of toric dynamical systems depend continuously  on the rate constants of the network~\cite{craciun2020structure}. Further, it is also known that many reaction networks that possess toric locus of measure zero; but a disguised toric locus having positive measure~\cite{2022disguised}. In particular, in previous work~\cite{disg_1}, we have established a lower bound on the dimension of the disguised toric locus; often, this lower bound is {\em equal} to the dimension of the parameter space of  $\mathcal N$, which is one way to conclude that the disguised toric locus has positive measure.

In this paper, we study several topological properties of the toric locus and the disguised toric locus. In particular, we establish the product structure of some relevant subsets of the toric locus  via certain homeomorphisms in Section~\ref{sec:homeomorphism}. Our main result is that the disguised toric locus is path connected (Theorem~\ref{thm:dis_toric}). This is useful since it can be exploited in numerical and homotopy methods for tracking the steady states of a system~\cite{bates2023numerical,breiding2018homotopycontinuation, collins2021singular, sommese2005numerical}. 

This work opens up avenues for a more comprehensive analysis of special regions of the parameter space of a reaction network; as a recent example of related work, where the focus is on instability/multistability as opposed to stability, see~\cite{telek_feliu_2023}.

\section*{Acknowledgements} 

This work was supported in part by the National Science Foundation grant DMS-2051568.

\bibliographystyle{unsrt}
\bibliography{Bibliography}

\end{document}